\documentclass[11pt]{amsart}
\usepackage{amsmath,amssymb,color,cite,verbatim}
\usepackage{etoolbox}
\apptocmd{\sloppy}{\hbadness 10000\relax}{}{}

\numberwithin{equation}{section}

\newtheorem{thm}[equation]{Theorem}
\newtheorem{prop}[equation]{Proposition}
\newtheorem{lemma}[equation]{Lemma}
\newtheorem{cor}[equation]{Corollary}

\theoremstyle{definition}
\newtheorem{rmk}[equation]{Remark}
\newtheorem{notation}[equation]{Notation}

\DeclareMathOperator{\charp}{char}
\DeclareMathOperator{\lcm}{lcm}

\DeclareMathOperator{\Aut}{Aut}

\newcommand{\mybar}[1]{#1\llap{$\overline{\phantom{\rm#1}}$}}
\newcommand{\abs}[1]{\lvert #1 \rvert}

\usepackage[colorlinks,pagebackref,pdftex, bookmarks=false]{hyperref}

\begin{document}

\title{Extensions of absolute values on two subfields}

\author{Zhiguo Ding}
\address{
  School of Advanced Study,
  Hunan Institute of Traffic Engineering,
  Hengyang, Hunan 421001 China
}
\email{ding8191@qq.com}

\author{Michael E. Zieve}
\address{
  Department of Mathematics,
  University of Michigan,
  530 Church Street,
  Ann Arbor, MI 48109-1043 USA
}
\email{zieve@umich.edu}
\urladdr{http://www.math.lsa.umich.edu/$\sim$zieve/}

\date{\today}

\begin{abstract}
We describe the absolute values on a field which simultaneously extend absolute values on two subfields.  We also give a common generalization of many versions of Abhyankar's lemma on ramification indices, which is both widely applicable and easy to state.  We then apply these results to count points on the fibered product of two curve morphisms $C\to X$ and $D\to X$ which lie over prescribed points on $C$ and $D$.
\end{abstract}

\thanks{
The authors thank Franz-Viktor Kuhlmann for helpful comments.
The second author thanks the National Science Foundation for support under grant DMS-1601844.}

\maketitle


\section{Introduction}

A classical topic in the theory of absolute values is the study of the set of absolute values on a field $F$ which extend a prescribed absolute value on a subfield $K$.  For instance, a classification of such extensions in case $F/K$ is finite appears in \cite[\S XII.3]{L}, and the case of finitely generated $F/K$ appears in \cite[\S 6.1, 7.1, 9.1]{C}.

In this paper we study absolute values on a field $F$ which extend prescribed absolute values on two subfields $L$ and $M$ of $F$.  As we will explain, this situation arises in many areas of math, including complex dynamics \cite{OZ}, diophantine equations \cite{BT}, functional equations \cite{AZ,R}, and value distribution of meromorphic functions \cite{SZ}.

\begin{notation}
Throughout this paper, $\widehat{K}$ denotes the completion of a field $K$ with respect to some prescribed absolute 
value, and $K^a$ denotes the algebraic closure of a field $K$.
\end{notation}

The statement of our first result uses the following standard facts. If $L/K$ and $M/K$ are 
field extensions with $L/K$ both finite and separable, then there is an isomorphism 
$\phi\colon L\otimes_K M\to F_1\times F_2\times\dots\times F_n$ for some integer $n>0$ and some fields 
$F_i$. Writing $\pi_i\colon F_1\times\dots\times F_n\to F_i$ for projection on the $i$-th coordinate, 
and $\iota\colon L\to L\otimes_K M$ for the canonical homomorphism, then $\pi_i\circ\phi\circ\iota$ is 
an injective homomorphism $L\to F_i$, and by identifying $L$ with its image we view $F_i$ as an extension 
of $L$. Likewise we view $F_i$ as an extension of $M$.

\begin{thm}\label{global}
Let $L,M,K,F_i$ be as above. Let $\abs{\,\cdot\,}_L$ and $\abs{\,\cdot\,}_M$ be absolute values on $L$ 
and $M$ which extend the same absolute value $\abs{\,\cdot\,}_K$ on $K$. Let $\mathcal{V}_i$ be the 
set of all absolute values on $F_i$ extending both $\abs{\,\cdot\,}_L$ and $\abs{\,\cdot\,}_M$, and for 
$\nu\in\mathcal{V}_i$ let $\widehat{F_{i,\nu}}$ be the completion of $F_i$ with respect to $\nu$. Then 
\begin{enumerate}
\item\label{count} there is a canonical bijection between the disjoint union $\coprod_{i=1}^n \mathcal{V}_i$ 
and the set of equivalence classes of all $\widehat{K}$-algebra embeddings $\widehat{L} \to \widehat{M}^a$ 
up to composition with $\Aut_{\widehat M}(\widehat{M}^a)$
\item\label{degree} we have 
$\sum_{i=1}^n \sum_{\nu\in\mathcal{V}_i} [\widehat{F_{i,\nu}} : \widehat{M}] = [\widehat{L} : \widehat{K}]$.
\end{enumerate}
\end{thm}

In Theorem~\ref{global} we have given a classification of the union of the sets $\mathcal{V}_i$. However, in some situations one needs control of each individual $\mathcal{V}_i$. Our next  
result resolves this question in a more general setting. The special case $L=M=K$ of the following result is classical. 

\begin{thm}\label{local}
Let $F/K$ be a field extension, and let $L, M$ be fields between $K$ and $F$ such that both $L/K$ 
and $F/M$ are finite. Let $\abs{\,\cdot\,}_L$ and $\abs{\,\cdot\,}_M$ be absolute values on $L$ and 
$M$ respectively which extend the same absolute value $\abs{\,\cdot\,}_K$ on $K$. Let $\mathcal{E}$ 
be the set of all $M$-algebra embeddings $\alpha\colon F \to \widehat{M}^a$ for which there exists a 
$\widehat K$-algebra embedding $\beta\colon\widehat L\to\widehat{M}^a$ such that $\alpha$ and $\beta$ 
have the same restriction to $L$. Then there is a canonical bijection between
\begin{enumerate}
\item the set of absolute values on $F$ which extend both $\abs{\,\cdot\,}_L$ and $\abs{\,\cdot\,}_M$, and
\item the set of equivalence classes of $\mathcal E$ up to composition with $\Aut_{\widehat M}(\widehat{M}^a)$.
\end{enumerate}
\end{thm}

Abhyankar's lemma on ramification indices in a compositum of field extensions has proven to be a valuable tool in many settings, and the literature contains many variants of this result.
We provide a common generalization of most such variants, which is both widely applicable and easy to state.
We recall the standard fact that if $N/K$ is a finite-degree field extension, and $K$ is complete with 
respect to a non-archimedean absolute value $\abs{\,\cdot\,}_K$, then there is a unique extension of 
$\abs{\,\cdot\,}_K$ to a non-archimedean absolute value $\abs{\,\cdot\,}_N$ on $N$. We write $\mybar{N}$ 
and $\mybar{K}$ for the residue fields of $N$ and $K$ with respect to $\abs{\,\cdot\,}_N$. For fields 
$U,V$ with $K\le V\le U\le N$, we write $e(U/V)$ for the ramification index of $U/V$ with respect to 
$\abs{\,\cdot\,}_N$, namely the index $[\abs{U^*}_N:\abs{V^*}_N]$ of $\abs{V^*}_N$ in $\abs{U^*}_N$ 
as subgroups of $\mathbb{R}^*$.

\begin{prop}\label{Abhyankar}
Let $N/K$ be a finite Galois extension, and let $L$, $M$, $F$ be fields between $K$ and $N$ such that 
$F = L.M$. Suppose $\abs{\,\cdot\,}_K$ is a complete discrete absolute value on $K$ such that $L/K$ 
is tamely ramified and $\mybar{N} / \mybar{K}$ is separable. Then $e(N/F) = \gcd(e(N/L), e(N/M))$. 
\end{prop}

We view Proposition~\ref{Abhyankar} as indicating the key reason why Abkyankar's lemma is true.  Although the form of Proposition~\ref{Abhyankar} differs from that of previous versions of Abhyankar's lemma, Proposition~\ref{Abhyankar} immediately implies the following result of a more familiar form:

\begin{cor}\label{old-abh-intro}
Let $F/K$ be a finite separable extension, and let $L,M$ be fields between $K$ and $F$ with $F = L.M$. 
Suppose $\abs{\,\cdot\,}_K$ is a complete discrete absolute value on $K$ such that $L/K$ is tamely ramified and the residue field of $K$ is perfect.
Then $e(F/K) = \lcm(e(L/K), e(M/K))$.
\end{cor}

Corollary~\ref{old-abh-intro} implies several versions of Abhyankar's lemma in the literature, including the cases of function fields \cite[Thm.~3.9.1]{St}, number fields \cite[p.~229]{Na}, and Dedekind domains \cite[Thm.~2.1]{CH}, as well as the number-theoretic result in \cite{T}.  There are also variants of Abhyankar's lemma in other settings, such as \cite[Lemme~3.6 in Exp.~10]{G} when the residue field extension may be inseparable, \cite[Lemma~15.105.4]{stack} for formally smooth $\mathfrak{m}$-adic topologies, \cite{DK} for transcendental extensions and henselizations, and
\cite{An} for perfectoid spaces.  Proposition~\ref{Abhyankar} has the desirable feature of being sufficiently general to apply to questions in many different branches of math, while also having a concise and easily understood statement and a short self-contained proof.  

Our results have the following geometric consequence.

\begin{cor} \label{df}
Let $\phi_1\colon C_1\to D$ and $\phi_2\colon C_2\to D$ be finite separable morphisms between smooth projective curves over an algebraically closed field $k$.  For any $S\in D(k)$, pick $P_j\in \phi_j^{-1}(S)$ such that at least one $P_j$ is tamely ramified over $S$.  Let $X_1,\dots,X_\ell$ be the components of the normalization of the fibered product $C_1\times_D C_2$, and let $\pi_{i,j}$ be the induced morphism $X_i\to C_j$ for $1\le i\le\ell$ and $1\le j\le 2$.  Then \[
\sum_{i=1}^{\ell}\, \bigl\lvert\pi_{i,1}^{-1}(P_1)\cap \pi_{i,2}^{-1}(P_2)\bigr\rvert=\gcd\bigl(e(P_1/S), e(P_2/S)\bigr).
\]
\end{cor}

Corollary~\ref{df} is useful in various contexts, since when combined with Abhyankar's lemma it describes the ramification of the projection maps $\pi_{i,j}$ in terms of the ramification of the maps $\phi_j$.  A recent application of this to value distribution of meromorphic functions is in \cite{SZ}.  Other applications of the resulting formula for the genus of $X_i$ in case each $\phi_j$ is tamely ramified are in \cite{AZ,BT,OZ,R} and other papers.  Many authors have proved special cases of Corollary~\ref{df}, for instance \cite{R,DW,F,DF}.  Compared to these previous results, the advantages of Corollary~\ref{df} are its greater generality and its short but fully detailed proof, especially regarding the case $\ell>1$ for which ours is the first rigorous proof in the literature.

This paper is organized as follows. In Section 2, we recall some standard facts about absolute 
values. In the next three sections we prove Theorem~\ref{global}, 
Theorem~\ref{local}, and Proposition~\ref{Abhyankar} respectively. We conclude in Section 6 with 
several consequences which generalize some known results in the literature, including in particular 
a generalization of a result counting points on the fibered product of two curve morphisms.


\section{Preliminaries}

In this section we recall some standard facts about absolute values which will be used in this paper; 
for the proofs see for instance \cite{C,L,S}. For the general theory of absolute values see any of
\cite{A,AM,B,BGR,C,CF,EP,L,Ne,Ribenboim,S,ZS}.

By a \emph{completion} of a field $K$ with respect to an absolute value $\abs{\,\cdot\,}_K$ we mean 
 a triple $(\widehat{K}, \abs{\,\cdot\,}_{\widehat{K}}, \iota)$ in which $\widehat{K}$ is a field, 
$\abs{\,\cdot\,}_{\widehat{K}}$ is an absolute value on $\widehat{K}$, and $\iota \colon K \to \widehat{K}$ 
is an embedding of fields, such that 
\begin{enumerate}
\item 
$\abs{\,\cdot\,}_{\widehat{K}}$ extends $\abs{\,\cdot\,}_K$ along the embedding 
$\iota \colon K \to \widehat{K}$, in the sense that
$\abs{x}_K = \abs{\iota(x)}_{\widehat{K}}$ for all $x\in K$,
\item $\widehat{K}$ is complete with respect to the absolute value $\abs{\,\cdot\,}_{\widehat{K}}$,
\item $K$ is dense in $\widehat{K}$ with respect to the topology induced by $\abs{\,\cdot\,}_{\widehat{K}}$.
\end{enumerate} 

There always exists a completion $(\widehat{K}, \abs{\,\cdot\,}_{\widehat{K}}, \iota)$ for any field $K$ with 
an absolute value $\abs{\,\cdot\,}_K$. Moreover, any completion $(\widehat{K}, \abs{\,\cdot\,}_{\widehat{K}}, \iota)$ 
has the \emph{universal property}: if $\abs{\,\cdot\,}_L$ is a complete absolute value on a field $L$, and 
$\rho \colon K \to L$ is an embedding of fields along which $\abs{\,\cdot\,}_L$ extends $\abs{\,\cdot\,}_K$, 
then there exists a unique embedding $\theta \colon \widehat{K} \to L$ which preserves both the absolute 
values and the $K$-algebra structures, in the sense that $\abs{x}_{\widehat{K}} = \abs{ \theta(x) }_L$ for all 
$x \in \widehat{K}$ and $\rho = \theta \circ \iota$. Therefore, any two completions are canonically isomorphic. 
For completions of absolute values see \cite[Prop.~XII.2.1]{L} or \cite[\S 2.4]{C}. 

The following result is essentially 
\cite[Prop.~XII.2.5 and XII.2.6]{L} and \cite[Thm.~7.1.1]{C}.

\begin{prop}\label{cext}
Suppose that $K$ is a field which is complete with respect to an absolute value $\abs{\,\cdot\,}_K$, and $L$ 
is an algebraic extension field over $K$. Then there exists a unique absolute value $\abs{\,\cdot\,}_L$ on 
$L$ which extends the absolute value $\abs{\,\cdot\,}_K$ on $K$. If in addition the degree $n := [L : K]$ 
is finite, then $L$ is complete with respect to $\abs{\,\cdot\,}_L$, and $\abs{x}_L = \abs{N_{L/K}(x)}_K^{1/n}$ 
for all $x \in L$ where $N_{L/K}$ is the norm from $L$ to $K$.
\end{prop}

By Proposition~\ref{cext} and the universal property of a completion, we have the following result which is \cite[Prop.~XII.3.1]{L}.

\begin{cor}\label{gen}
Suppose $L/K$ be a finite extension of fields, $\abs{\,\cdot\,}_K$ is an absolute value on $K$, and 
$\abs{\,\cdot\,}_L$ is an extension of $\abs{\,\cdot\,}_K$ to $L$. Then $\widehat{L} = L. \widehat{K}$.
\end{cor}

Let $\abs{\,\cdot\,}$ be a non-archimedean absolute value on a field $K$. Then the set 
$\mathfrak{O} := \{ x\in K \colon \abs{x} \le 1 \}$ forms a valuation ring of $K$ with the unique 
maximal ideal $\mathfrak{P} := \{ x\in K \colon \abs{x} < 1 \}$. We say that $\mathfrak{O}$ is the 
\emph{ring of integers} in $K$ and the quotient $\mybar{K} := \mathfrak{O} / \mathfrak{P}$ is the 
\emph{residue field} of $K$. 

Suppose $L/K$ is any extension of fields, $\abs{\,\cdot\,}_K$ is an absolute value on $K$, and 
$\abs{\,\cdot\,}_L$ is an extension of $\abs{\,\cdot\,}_K$ to $L$. The \emph{ramification index} $e$ 
is the index $[\,\abs{L^*}_L : \abs{K^*}_K\,]$ of $\abs{K^*}_K$ in $\abs{L^*}_L$ as subgroups of the 
multiplicative group $\mathbb{R}^*$. Suppose in addition $\abs{\,\cdot\,}_K$ is non-archimedean, then 
$\abs{\,\cdot\,}_L$ is also non-archimedean. The \emph{residue degree} $f$ is the degree of the 
residue field extension $\mybar{L} / \mybar{K}$. Thus $e$ and $f$ are either positive integers or $\infty$.
Moreover, we have

\begin{prop}\label{e&f}
Suppose $L/K$ be a finite extension of fields, $\abs{\,\cdot\,}_K$ is a non-archimedean absolute value 
on $K$, and $\abs{\,\cdot\,}_L$ is an absolute value on $L$ which extends $\abs{\,\cdot\,}_K$ on $K$. Then 
\begin{enumerate}
\item\label{1st} the relation $[L : K] \ge e(L/K) \cdot f(L/K)$ holds, in particular both $e(L/K)$ and 
$f(L/K)$ are positive integers. Furthermore, the equality $[L : K] = e(L/K) \cdot f(L/K)$ holds if 
$\abs{\,\cdot\,}_K$ is discrete and complete, 
\item\label{2nd} $e(L/K) = e(\widehat{L}/\widehat{K})$ and $f(L/K) = f(\widehat{L}/\widehat{K})$, where 
$\widehat{K}$ and $\widehat{L}$ are completions of $K$ and $L$ with respect to $\abs{\,\cdot\,}_K$ and 
$\abs{\,\cdot\,}_L$ respectively.
\end{enumerate}
\end{prop}

Part~\eqref{1st} of Proposition~\ref{e&f} is \cite[Prop.~XII.4.6 and XII.6.1]{L}; see also
\cite[Thm.~7.5.1]{C}  for the equality in case 
$\abs{\,\cdot\,}_K$ is both discrete and complete. For part~\eqref{2nd} of Proposition~\ref{e&f} 
see \cite[\S XII.5]{L}.

Suppose $K$ is a field complete with respect to a nontrivial discrete absolute value $\abs{\,\cdot\,}_K$, 
and $L/K$ is a finite Galois extension with the Galois group $G(L/K)$. By Proposition~\ref{cext} there 
exists a unique absolute value $\abs{\,\cdot\,}_L$ on $L$ extending $\abs{\,\cdot\,}_K$. Furthermore, 
$\abs{\,\cdot\,}_L$ is discrete and $L$ is complete with respect to it. For any integer $i\ge -1$, 
the \emph{$i$-th ramification group} $G_i(L/K)$ is the subgroup of the Galois group $G(L/K)$ consisting 
of all $\sigma\in G(L/K)$ such that $\sigma(x) - x \in \mathfrak{P}_L^{i+1}$ for any $x\in \mathfrak{O}_L$. 
The groups $G_i(L/K)$ form a decreasing chain of normal subgroups of $G(L/K)$; $G_{-1}(L/K) = G(L/K)$, 
$G_0(L/K)$ is the \emph{inertia subgroup} of $G(L/K)$, and $G_i(L/K) = 1$ for sufficiently large $i$. 
For the basic knowledge of ramification groups see \cite[\S IV]{S} or \cite[\S 8.7]{C}.

The next result follows from \cite[\S I.20]{S} and the results of \cite[Ch.~IV]{S};
for a quick treatment see \cite[Thm.~8.8.1]{C}.

\begin{prop}\label{ramgp}
Suppose $K$ is a field complete with respect to a nontrivial discrete absolute value $\abs{\,\cdot\,}_K$, and 
$L/K$ is a finite Galois extension whose residue field extension $\mybar{L} / \mybar{K}$ is separable. Then
\begin{enumerate}
\item $\mybar{L} / \mybar{K}$ is a finite Galois extension, whose Galois group $G(\mybar{L}/\mybar{K})$ is 
canonically isomorphic to $G(L/K) / G_0(L/K)$. Thus the ramification index $e(L/K)$ equals the order of the 
inertia group $G_0(L/K)$, \item $G_0(L/K) / G_1(L/K)$ is canonically isomorphic to some multiplicative subgroup 
of $\mybar{L}$, $G_i(L/K) / G_{i+1}(L/K)$ is non-canonically isomorphic to some additive subgroup of $\mybar{L}$ 
for any integer $i\ge 1$, \item therefore, if $\charp(\mybar{K}) = 0$, then $G_1(L/K) = 1$ and $G_0(L/K)$ is 
a finite cyclic group; if $\charp(\mybar{K}) = p > 0$, then $G_1(L/K)$ is the unique Sylow $p$-subgroup of 
$G_0(L/K)$ and $G_0(L/K) / G_1(L/K)$ is a finite cyclic group of order coprime to $p$.
\end{enumerate}
\end{prop}


\section{Proof of Theorem~\ref{global}}

Throughout this section, we assume $L/K$ and $M/K$ are field extensions with $L/K$ both finite and separable, and also
$\abs{\,\cdot\,}_L$ and $\abs{\,\cdot\,}_M$ are absolute values on $L$ and $M$, respectively, which extend the same 
absolute value $\abs{\,\cdot\,}_K$ on $K$.

Since $L/K$ is finite separable, there exists an element $s \in L$ such that $L = K(s)$. Let $f(X) \in K[X]$ 
be the irreducible polynomial of $s$ over $K$, which is separable since $L/K$ is separable. Then $L = K(s)$ 
is naturally isomorphic to the quotient ring $K[X]/(f)$ of the polynomial ring $K[X]$ by the ideal $(f)$. 
Write $f(X) = f_1(X) f_2(X) \cdots f_n(X)$ in $M[X]$, where $n\ge 1$ and $f_i(X)$ are distinct monic irreducible 
polynomials over $M$. Hence, the tensor product $L \otimes_K M$ is naturally isomorphic to $M[X]/(f)$, which by 
Chinese remainder theorem is naturally isomorphic to a product of fields $F_1 \times F_2 \times \cdots \times F_n$, 
where $F_i := M[X]/(f_i)$ for any $1\le i\le n$. 

Via the $i$-th projection from $F_1 \times F_2 \times \cdots \times F_n$, each $F_i$ comes equipped with 
embeddings 
$L \to F_i$ and $M \to F_i$ which yield a commutative diagram with the given embeddings $K \to L$ and $K \to M$. 
In other words, the two embeddings $K \to L \to F_i$ and  $K \to M \to F_i$ are identical for any $i$

By the universal property of $\widehat{K}$, there exist unqiue embeddings $\widehat{K} \to \widehat{L}$ and 
$\widehat{K} \to \widehat{M}$ which preserve both the absolute values and the $K$-algebra structures. By 
Corollary~\ref{gen} we know $\widehat{L} = L. \widehat{K}$, or equivalently, $\widehat{L} = \widehat{K}(s)$.

Let $g(X) \in \widehat{K}[X]$ be the irreducible polynomial of $s$ over $\widehat{K}$, so $g(X)$ is a monic 
irreducible separabe factor of $f(X)$ over $\widehat{K}$. Then $\widehat{L} = \widehat{K}(s)$ is naturally 
isomorphic to the quotient ring $\widehat{K}[X]/(g)$ of the polynomial ring $K[X]$ by the ideal $(g)$. Write 
$g(X) = g_1(X) g_2(X) \cdots g_m(X)$ in $\widehat{M}[X]$, where $m\ge 1$ and $g_j(X)$ are distinct monic 
irreducible polynomials over $\widehat{M}$. Hence, the tensor product 
$\widehat{L} \otimes_{\widehat{K}} \widehat{M}$ is naturally isomorphic to $\widehat{M}[X]/(g)$, 
which by Chinese remainder theorem is naturally isomorphic to a product of fields 
$G_1 \times G_2 \times \cdots \times G_m$, where $G_j := \widehat{M}[X]/(g_j)$ for any $1\le j\le m$. 

Via the $j$-th projection from $G_1 \times G_2 \times \cdots \times G_m$, each $G_j$ comes equipped with 
embeddings $\widehat{L} \to G_j$ and $\widehat{M} \to G_j$ which yield a commutative diagram with the 
given embeddings $\widehat{K} \to \widehat{L}$ and $\widehat{K} \to \widehat{M}$. By Proposition~\ref{cext} 
there exists a unique absolute value $\abs{\,\cdot\,}_j$ on $G_j$ which extends the absolute value 
$\abs{\,\cdot\,}_{\widehat{M}}$ on $\widehat{M}$. Moreover, $G_j$ is complete with respect to 
$\abs{\,\cdot\,}_j$, and the restriction $\abs{\,\cdot\,}_j$ to $\widehat{L}$ is equal to 
$\abs{\,\cdot\,}_{\widehat{L}}$. 

\begin{lemma}
There exists a unique map 
\[
\sigma \colon \{1,2,\dots,m\} \to \{1,2,\dots,n\}
\] 
such that each $g_j(X)$ is a monic irreducible factor of $f_{\sigma(j)}(X)$ over $\widehat{M}$. 
\end{lemma}

\begin{proof}
On the one hand, each $g_j(X) \in \widehat{M}[X]$ 
is a monic irreducible factor of $g(X) \in \widehat{K}[X]$ which divides $f(X) \in K[X]$; on the other hand, 
we know that $f(X) = f_1(X) f_2(X) \cdots f_n(X)$ in $M[X]$, where $f_i(X)$ are distinct monic irreducible 
polynomials over $M$. Therefore, each $g_j(X)$ divides some unique $f_i(X)$ which concludes the proof.
\end{proof}

For any $1\le j\le m$, there exists a natural embedding $\iota_j \colon F_{\sigma(j)} \to G_j$ which preserves the 
structures of both $L$-alegbras and $M$-algebras. Hence, via this embedding $\iota_j \colon F_{\sigma(j)} \to G_j$, 
the absolute value $\abs{\,\cdot\,}_j$ on $G_j$ gives rise to an absolute value $\abs{\,\cdot\,}^j_{\sigma(j)}$ on 
$F_{\sigma(j)}$ which extends both $\abs{\,\cdot\,}_L$ and $\abs{\,\cdot\,}_M$ on $L$ and $M$ respectively. Moreover, 
it is clear that $F_{\sigma(j)}$ is dense in $G_j$ with respect to the topology induced by $\abs{\,\cdot\,}_j$. 
Therefore, we get 

\begin{lemma}\label{comp}
For any $1\le j\le m$, the triple $(G_j, \abs{\,\cdot\,}_j, \iota_j)$ is a completion of $F_{\sigma(j)}$ with 
respect to the restriction $\abs{\,\cdot\,}^j_{\sigma(j)}$ of $\abs{\,\cdot\,}_j$ on $F_{\sigma(j)}$.  
\end{lemma}

Moreover, we have the following

\begin{lemma}\label{bijection}
For any $1\le i\le n$, the map $j \mapsto \abs{\,\cdot\,}^j_{\sigma(j)}$ gives rise to a bijection $\Psi_i$ 
from the fiber $\sigma^{-1}(i)$ to the set $\mathcal{V}_i$ of all absolute values on $F_i$ which extend both 
$\abs{\,\cdot\,}_L$ and $\abs{\,\cdot\,}_M$ on $L$ and $M$ respectively.
\end{lemma}

\begin{proof}
First, suppose two indices $j$ and $j'$ in the fiber $\sigma^{-1}(i)$ induce the same absolute value 
$\abs{\,\cdot\,}^j_i = \abs{\,\cdot\,}^{j'}_i$ on the field $F_i$. Then by Lemma~\ref{comp} both 
$(G_j, \abs{\,\cdot\,}_j, \iota_j)$ and $(G_{j'}, \abs{\,\cdot\,}_{j'}, \iota_{j'})$ are completions of $F_i$ 
with respect to the same absolute value $\abs{\,\cdot\,}^j_i = \abs{\,\cdot\,}^{j'}_i$. Hence, there exists 
a unique isomorphism $\theta \colon G_j \to G_{j'}$ which preserves both the absolute values and the 
$F_i$-algebra structures. In particular, $\theta \colon G_j \to G_{j'}$ is an $M$-algebra isomorphism, 
so $\theta \colon G_j \to G_{j'}$ is an $\widehat{M}$-algebra isomorphism by the universal property of 
$\widehat{M}$. Note that $\theta$ sends $\mybar{X}$ in $G_j := \widehat{M}[X]/(g_j)$ to $\mybar{X}$ in 
$G_{j'} := \widehat{M}[X]/(g_{j'})$, hence we get $g_j(X) = g_{j'}(X)$, or equivalently $j=j'$, since $g_j(X)$ 
and $g_{j'}(X)$ are the irreducible polynomials over $\widehat{M}$ of $\mybar{X}$ in $G_j$ and $\mybar{X}$ 
in $G_{j'}$ respectively. 

Conversely, suppose $\abs{\,\cdot\,}$ is an absolute values on $F_i$ extending both $\abs{\,\cdot\,}_L$ 
and $\abs{\,\cdot\,}_M$ on $L$ and $M$ respectively. Denote by $\widehat{F_i}$ the completion of $F_i$ with 
respect to $\abs{\,\cdot\,}$. By the universal property of $\widehat{L}$, there exists a unique embedding 
$\widehat{L} \to \widehat{F_i}$ which preserves both the absolute values and $L$-algebra structures. 
Similarly, there exists a unique embedding $\widehat{M} \to \widehat {F_i}$ preserving both the absolute 
values and $M$-algebra structures. By Corollary~\ref{gen} we know $\widehat{F_i} = \widehat{M}(y)$ where 
$y$ is the image of $\mybar{X} \in F_i := M[X]/(f_i)$ in $\widehat{F_i}$. Moreover, the irreducible 
polynomial of $y$ over $\widehat{M}$ must be $g_j(X)$ for some $j\in \sigma^{-1}(i)$, since it divides 
both the irreducible polynomial $f_i(X)$ of $\mybar{X} \in F_i$ over $M$ and the irreducible polynomial 
$g(X)$ of $s$ over $\widehat{K}$. Hence, there exists a natural $\widehat{M}$-algebra isomorphism 
$\widehat{F_i} = \widehat{M}(y) \to G_j := \widehat{M}[X]/(g_j)$, which preserves the absolute values 
by Proposition~\ref{cext}. Therefore, the given absolute value $\abs{\,\cdot\,}$ on $F_i$ is equal to 
the absolute value $\Psi_i(j) = \abs{\,\cdot\,}^j_i$ for some $j\in \sigma^{-1}(i)$, which concludes 
the proof. 
\end{proof}

Now we are ready to prove Theorem~\ref{global}.

\begin{proof}[Proof of Theorem~\ref{global}]
One the one hand, by Lemma~\ref{bijection} we can construct a bijection 
$\Psi \colon \{1, 2, \dots, m\} \to \coprod_{i=1}^n \mathcal{V}_i$ whose restriction on each fiber $\sigma^{-1}(i)$ 
coincides with the bijection $\Psi_i \colon \sigma^{-1}(i) \to \mathcal{V}_i$. On the other hand, there exists 
a bijection from the set of all roots of $g(X)$ in $\widehat{M}^a$ to the set of all $\widehat{K}$-algebra 
embeddings $\widehat{L} \to \widehat{M}^a$, which induces a bijection from the set $\{1, 2, \dots, m\}$ to 
the set of equivalence classes of all $\widehat{K}$-algebra embeddings $\widehat{L} \to \widehat{M}^a$ up to 
composition with $\Aut_{\widehat M}(\widehat{M}^a)$. Therefore, there is a bijection between the disjoint union 
$\coprod_{i=1}^n \mathcal{V}_i$ and the set of equivalence classes of all $\widehat{K}$-algebra embeddings 
$\widehat{L} \to \widehat{M}^a$ up to composition with $\Aut_{\widehat M}(\widehat{M}^a)$, which concludes the 
proof of the part~\eqref{count}. 

Furthermore, by Lemma~\ref{bijection} and Lemma~\ref{comp} we have 
\begin{align*}
\sum_{i=1}^n \sum_{\nu\in\mathcal{V}_i} [\widehat{F_i,\nu} : \widehat{M}] &= \sum_{i=1}^n \sum_{j\in \sigma^{-1}(i)} 
[G_j : \widehat{M}] = \sum_{i=1}^n \sum_{j\in \sigma^{-1}(i)} \deg(g_j) \\ &= \deg(g) = [\widehat{L} : \widehat{K}],
\end{align*}
which concludes the proof of the part~\eqref{degree}. 
\end{proof}

\begin{rmk}
Theorem~\ref{global} is independent of the choice of the generator $s$ over $K$ of the finite separable extension 
$L/K$. Indeed, the way to express the tensor product $L \otimes_K M$ as $F_1 \times F_2 \times \cdots \times F_n$ 
is unique up to isomorphism except the order of $F_i$, where each $F_i$ is a field containing both $L$ and $M$.
Moreover, up to this isomorphsim, the bijection is canonical between the disjoint union $\coprod_{i=1}^n \mathcal{V}_i$ 
and the set of equivalence classes of all $\widehat{K}$-algebra embeddings $\widehat{L} \to \widehat{M}^a$ up to 
composition with $\Aut_{\widehat M}(\widehat{M}^a)$.
\end{rmk}


\section{Proof of Theorem~\ref{local}}

Throughout this section, we assume $F/K$ is a field extension, $L$ and $M$ are subextensions of $F$ over $K$ 
such that both $L/K$ and $F/M$ are finite, $\abs{\,\cdot\,}_L$ and $\abs{\,\cdot\,}_M$ are absolute values on $L$ 
and $M$ respectively which extend the same absolute value $\abs{\,\cdot\,}_K$ on $K$. 

Since $L/K$ is a finite extension, we can write $L = K(S)$ where $S\subseteq L$ is a finite set which is 
algebraic over $K$. Similarly, we can write $F = M(T)$ where $T\subseteq F$ is a finite set which is algebraic 
over $M$. By the universal property of $\widehat{K}$ there exist unique embeddings $\widehat{K} \to \widehat{L}$ 
and $\widehat{K} \to \widehat{M}$ which preserve both the absolute values and the $K$-algebra structures. By 
Corollary~\ref{gen} we know $\widehat{L} = \widehat{K}(S)$, so $\widehat{L} / \widehat{K}$ is a finite extension. 

Let $\mathcal{E}$ be the set of all $M$-algebra embeddings $\alpha\colon F \to \widehat{M}^a$ for which there 
exists a $\widehat K$-algebra embedding $\beta\colon\widehat L\to\widehat{M}^a$ such that $\alpha$ and $\beta$ 
have the same restriction to $L$. Our goal is to show that there exists a canonical bijection between the set 
$\mathcal{V}$ of all absolute values on $F$ extending both $\abs{\,\cdot\,}_L$ and $\abs{\,\cdot\,}_M$ and the 
set of equivalence classes of $\mathcal E$ up to composition with $\Aut_{\widehat M}(\widehat{M}^a)$.

\begin{lemma}\label{map}
There exists a canonical map $\Phi \colon \mathcal{E} \to \mathcal{V}$ which sends any embedding 
$\alpha \colon F \to \widehat{M}^a$ in $\mathcal{E}$ to the restriction of $\abs{\,\cdot\,}_{\widehat{M}^a}$ on 
$F$ along $\alpha$.
\end{lemma}

\begin{proof}
For any embedding $\alpha\colon F \to \widehat{M}^a$ in $\mathcal{E}$, by definition there exists some 
$\widehat{K}$-algebra embedding $\beta \colon \widehat{L} \to \widehat{M}^a$ such that $L \to F \to \widehat{M}^a$ 
is equal to $L \to \widehat{L} \to \widehat{M}^a$. By Proposition~\ref{cext} the restriction of 
$\abs{\,\cdot\,}_{\widehat{M}^a}$ on $\widehat{L}$ along $\beta$ and the absolute value $\abs{\,\cdot\,}_{\widehat{L}}$ 
are equal since both of them extend $\abs{\,\cdot\,}_{\widehat{K}}$ and $\widehat{L} / \widehat{K}$ is finite. Hence, 
the restriction of $\abs{\,\cdot\,}_{\widehat{M}^a}$ on $F$ along $\alpha$ is an absolute value on $F$ which extends 
both extending both $\abs{\,\cdot\,}_L$ and $\abs{\,\cdot\,}_M$ on $L$ and $M$ respectively. Therefore, there is 
a canonical map $\Phi \colon \mathcal{E} \to \mathcal{V}$ which sends any embedding $\alpha \colon F \to \widehat{M}^a$ 
in $\mathcal{E}$ to the restriction of $\abs{\,\cdot\,}_{\widehat{M}^a}$ on $F$ along $\alpha$.
\end{proof}

\begin{lemma}\label{surjective}
The map $\Phi \colon \mathcal{E} \to \mathcal{V}$ is surjective.
\end{lemma}

\begin{proof}
Suppose $\abs{\,\cdot\,}_F$ is an absolute value in $\mathcal{V}$, in other words, $\abs{\,\cdot\,}_F$ is an absolute 
value on $F$ which extends both $\abs{\,\cdot\,}_L$ and $\abs{\,\cdot\,}_M$. Denote by $\widehat{F}$ the completion of 
$F$ with respect to $\abs{\,\cdot\,}_F$. By the universal property of $\widehat{L}$, there exists a unique $L$-algebra 
embedding $\widehat{L} \to \widehat{F}$ preserving the absolute values. Similarly, there exists a unique $M$-algebra 
embedding $\widehat{M} \to \widehat{F}$ preserving the absolute values. Since $F = M(T)$ is a finite extension over 
$M$, by Corollary~\ref{gen} we have $\widehat{F} = \widehat{M}(T)$, so $\widehat{F}$ is a finite field extension 
over $\widehat{M}$. Hence, there exists an $\widehat{M}$-algebra embedding $\widehat{F} \to \widehat{M}^a$ along which 
the absolute values are preserved. Denote by $\alpha$ the composition $F \to \widehat{F} \to \widehat{M}^a$ and denote 
by $\beta$ the composition $\widehat{L} \to \widehat{F} \to \widehat{M}^a$. Then $\alpha\colon F \to \widehat{M}^a$ 
is an $M$-algebra embedding, $\beta \colon \widehat{L} \to \widehat{M}^a$ is a $\widehat{K}$-algebra embedding, and 
$L \to F \to \widehat{M}^a$ equals $L \to \widehat{L} \to \widehat{M}^a$. Therefore, the embedding 
$\alpha\colon F \to \widehat{M}^a$ lies in $\mathcal{E}$ and its image $\Phi(\alpha)$ is the given absolute value 
$\abs{\,\cdot\,}_F$ in $\mathcal{V}$.
\end{proof}

\begin{lemma}\label{fiber}
Two embeddings $\alpha_1$ and $\alpha_2$ in $\mathcal{E}$ have the same image under the map 
$\Phi \colon \mathcal{E} \to \mathcal{V}$ if and only if $\alpha_2 = \tau \circ \alpha_1$ for 
some $\tau \in \Aut_{\widehat M}(\widehat{M}^a)$.
\end{lemma}

\begin{proof}
Let $\alpha_1$ and $\alpha_2$ be any two embeddings in $\mathcal{E}$. First, suppose $\alpha_2 = \tau \circ \alpha_1$ 
for some $\widehat{M}$-algebra isomorphism $\tau \colon \widehat{M}^a \to \widehat{M}^a$. By Proposition~\ref{cext} 
the $\widehat{M}$-algebra isomorphism $\tau \colon \widehat{M}^a \to \widehat{M}^a$ preserves the absolute values, 
so the identity $\alpha_2 = \tau \circ \alpha_1$ implies that $\alpha_1$ and $\alpha_2$ give the same absolute 
value $\abs{\,\cdot\,}_F$ on $F$. Therefore, $\alpha_1$ and $\alpha_2$ have the same image under the map 
$\Phi \colon \mathcal{E} \to \mathcal{V}$. Conversely, suppose $\alpha_1$ and $\alpha_2$ have the same image 
$\abs{\,\cdot\,}_F$ under the map $\Phi \colon \mathcal{E} \to \mathcal{V}$. For any $i\in \{1,2\}$, 
$\widehat{M}(\alpha_i(T))$ is finite over $\widehat{M}$ since $F = M(T)$ is finite over $M$, so by 
Proposition~\ref{cext} $\widehat{M}(\alpha_i(T))$ is complete with respect to the restricted absolute value 
$\abs{\,\cdot\,}_i$ of $\abs{\,\cdot\,}_{\widehat{M}^a}$, hence $(\widehat{M}(\alpha_i(T)), \abs{\,\cdot\,}_i, \alpha_i)$ 
is a completion of $F$ with respect to $\abs{\,\cdot\,}_F$. By the universal properties of completions, there exists 
a unique isomorphism $\tau \colon \widehat{M}(\alpha_1(T)) \to \widehat{M}(\alpha_2(T))$ which preserves both the 
absolute values and the $F$-algebra structures. In particular, 
$\tau \colon \widehat{M}(\alpha_1(T)) \to \widehat{M}(\alpha_2(T))$ preserves the $M$-algebra structures, which 
implies it preserves the $\widehat{M}$-algebra structures by the universal property of $\widehat{M}$. Hence, 
$\tau \colon \widehat{M}(\alpha_1(T)) \to \widehat{M}(\alpha_2(T))$ is an $\widehat{M}$-algebra isomorphism. Thus 
we can extend $\tau$ to an $\widehat{M}$-algebra automorphism of $\widehat{M}^a$, which we denote also by $\tau$ 
by abuse of language. Therefore, $\alpha_2 = \tau \circ \alpha_1$ for some $\widehat{M}$-algebra isomorphism 
$\tau \colon \widehat{M}^a \to \widehat{M}^a$.
\end{proof}

Now, we are ready to prove Theorem~\ref{local}.

\begin{proof}[Proof of Theorem~\ref{local}]
By the combination of Lemma~\ref{map}, Lemma~\ref{surjective}, and Lemma~\ref{fiber}, we know that the map 
$\Phi \colon \mathcal{E} \to \mathcal{V}$ induces a canonical bijection from the set of equivalence classes of 
$\mathcal E$ up to composition with $\Aut_{\widehat M}(\widehat{M}^a)$ to the set $\mathcal{V}$ of all absolute 
values on $F$ which extends both $\abs{\,\cdot\,}_L$ and $\abs{\,\cdot\,}_M$.
\end{proof}


\section{Proof of Proposition~\ref{Abhyankar}}

In this section, we use Proposition~\ref{ramgp} to prove Proposition~\ref{Abhyankar}.

\begin{proof}[Proof of Proposition~\ref{Abhyankar}]
Since $N/K$ is finite Galois and $\mybar{N}/\mybar{K}$ is separable, we know that $N$ is finite Galois over any 
of $F$, $L$, $M$, and $\mybar{N}$ is separable over any of $\mybar{F}$, $\mybar{L}$, $\mybar{M}$. So we have 
$e(N/F) = \abs{G_0(N/F)}$, $e(N/L) = \abs{G_0(N/L)}$, and $e(N/M) = \abs{G_0(N/M)}$. Therefore, the conclusion 
\[
e(N/F) = \gcd(e(N/L), e(N/M))
\]
is equivalent to the statement 
\[
\abs{G_0(N/F)} = \gcd(\abs{G_0(N/L)}, \abs{G_0(N/M)}).
\]
By the assumption $F = L.M$ we know $G(N/F) = G(N/L) \cap G(N/M)$, thus $G_i(N/F) = G_i(N/L) \cap G_i(N/M)$ 
holds for any $i$. 

If $\charp(\mybar{K}) = 0$, then $G_0(N/K)$ is a finite cyclic group, so we have
\[
\abs{G_0(N/F)} = \abs{G_0(N/L) \cap G_0(N/M)} = \gcd(\abs{G_0(N/L)}, \abs{G_0(N/M)}),
\]
which concludes the proof.

If $\charp(\mybar{K}) = p > 0$, then $G_1(N/K)$ is the unique Sylow $p$-subgroup of $G_0(N/K)$ 
and the quotient $G_0(N/K) / G_1(N/K)$ is a finite cyclic group of order coprime to $p$. Let 
$\Phi \colon G_0(N/K) \to G_0(N/K) / G_1(N/K)$ be the canonical surjective homomorphism of groups.
Therefore, we have
\[
\abs{G_0(N/K)} = \abs{G_1(N/K)} \times \abs{G_0(N/K) / G_1(N/K)}
\]
in which 
$\abs{G_1(N/K)}$ and $\abs{G_0(N/K) / G_1(N/K)}$ are coprime. By restricting $\Phi$ to $G_0(N/F)$,
we get a surjective homomorphism $G_0(N/F) \to \Phi(G_0(N/F))$ with the kernel $G_1(N/F)$. 
Hence, we have 
\[
\abs{G_0(N/F)} = \abs{G_1(N/F)} \times \abs{\Phi(G_0(N/F))}.
\]
Similarly, we have 
\[
\abs{G_0(N/L)} = \abs{G_1(N/L)} \times \abs{\Phi(G_0(N/L))}
\]
and 
\[
\abs{G_0(N/M)} = \abs{G_1(N/M)} \times \abs{\Phi(G_0(N/M))}.
\]
Therefore, in order to show the desired statement 
\[
\abs{G_0(N/F)} = \gcd(\abs{G_0(N/L)}, \abs{G_0(N/M)})
\]
it is equvalent to show both 
\[
\abs{G_1(N/F)} = \gcd(\abs{G_1(N/L)}, \abs{G_1(N/M)})
\]
and 
\[
\abs{\Phi(G_0(N/F))} = \gcd(\abs{\Phi(G_0(N/L))}, \abs{\Phi(G_0(N/M))}).
\]
Since $L/K$ is tamely ramified, the ramification index $e(L/K)$ is coprime to $p$, so $G_0(N/L)$
contains the unique Sylow $p$-subgroup $G_1(N/K)$ of $G_0(N/K)$, which implies 
$G_1(N/L) = G_1(N/K)$ and $G_1(N/F) = G_1(N/M)$. Hence, we have shown the part 
\[
\abs{G_1(N/F)} = \gcd(\abs{G_1(N/L)}, \abs{G_1(N/M)}).
\]
Since the quotient $G_0(N/K) / G_1(N/K)$ is a finite cyclic group, we have 
\[
\abs{\Phi(G_0(N/L)) \cap \Phi(G_0(N/M))} = \gcd(\abs{\Phi(G_0(N/L))}, \abs{\Phi(G_0(N/M))}).
\]
Hence, the remaining part 
\[
\abs{\Phi(G_0(N/F))} = \gcd(\abs{\Phi(G_0(N/L))}, \abs{\Phi(G_0(N/M))})
\]
is equivalent to the statement 
\[
\Phi(G_0(N/F)) = \Phi(G_0(N/L)) \cap \Phi(G_0(N/M)).
\]
Clearly, $\Phi(G_0(N/F))$ is a subgroup of both $\Phi(G_0(N/L))$ and $\Phi(G_0(N/M))$, 
so $\Phi(G_0(N/F))$ is contained in $\Phi(G_0(N/L)) \cap \Phi(G_0(N/M))$. Conversely,
any element in $\Phi(G_0(N/L)) \cap \Phi(G_0(N/M))$ can be written as $\Phi(\sigma) = \Phi(\tau)$
for some $\sigma \in G_0(N/L)$ and some $\tau \in G_0(N/M)$. Hence,
\[
\tau \in \sigma \cdot G_1(N/K) = \sigma \cdot G_1(N/L) \subseteq G_0(N/L)
\]
which implies 
\[
\tau \in G_0(N/L) \cap G_0(N/M) = G_0(N/F).
\]
Therefore, we get $\Phi(\tau) \in \Phi(G_0(N/F))$, which concludes the proof.
\end{proof}


\section{Consequences}

In this section, we give some consequences of our results. 

\begin{cor}\label{incext}
Suppose $F/K$ is a finite extension of fields and $\abs{\,\cdot\,}_K$ is an absolute value on $K$. Then there 
exists a canonical bijection between the set $\mathcal{V}$ of all absolute values on $F$ extending $\abs{\,\cdot\,}_K$ 
and the set of equivalence classes of all $K$-algebra embeddings $F \to \widehat{K}^a$ up to composition with 
$\Aut_{\widehat M}(\widehat{M}^a)$.
\end{cor}

\begin{proof}
It follows directly from Theorem~\ref{local} by taking $L = M = K$, since in this special case the set $\mathcal{E}$ 
defined in Theorem~\ref{local} is reduced simply to the set of all $K$-algebra embeddings $F \to \widehat{K}^a$.
\end{proof}

\begin{rmk}
Corollary~\ref{incext} is proved in \cite[Prop.~XII.3.1 and XII.3.2]{L}. In case $F/K$ is separable, Corollary~\ref{incext} is essentially 
\cite[Lemma~9.2.1]{C}, which directly implies \cite[Thm.~9.1.1]{C}; see also \cite[Thm.\ on p.~57]{CF}.
\end{rmk}

\begin{cor}\label{old-abh}
Let $F/K$ be a finite separable extension, and let $L$, $M$ be subextensions of $F$ over $K$ with $F = L.M$. 
Suppose $\abs{\,\cdot\,}_K$ is a complete discrete absolute value on $K$ such that $L/K$ is tamely ramified and 
$\mybar{K}$ is perfect. Then $e(F/K) = \lcm\bigl(e(L/K), e(M/K)\bigr)$.
\end{cor}

\begin{proof}
Let $N$ be the Galois closure of $F/K$. Since the residue field $\mybar{K}$ is perfect, by Proposition~\ref{Abhyankar} 
we have $e(N/F) = \gcd\bigl(e(N/L), e(N/M)\bigr)$, which is equivalent to the desired conclusion $e(F/K) = \lcm\bigl(e(L/K), e(M/K)\bigr)$.
\end{proof}

\begin{rmk}
Both the proof of Proposition~\ref{Abhyankar} and the proof of \cite[Thm.~2.1]{CH} are 
similar to the proof of \cite[Thm.~3.9.1]{St}.
We emphasize that Abhyankar's lemma is a local property, so that our version in Proposition~\ref{Abhyankar} is not only more 
general and more concise than many versions, but also its statement seems closer to isolating the key underlying mechanism.
\end{rmk}

We conclude this paper with an application of our results to nonsingular projective curves 
over an algebraically closed field.

\begin{notation}\label{note}
Let $L/K$ and $M/K$ be finite extensions of function fields over an algebraically closed field $k$, where at 
least one of $L/K$ and $M/K$ is separable.  Let $S$ be a place of $K$ and let $P$ and $R$ be places of $L$ and $M$, respectively, which 
both lie over $S$. Then the tensor product $L \otimes_K M$ is a product of fields 
$F_1 \times F_2 \times \cdots \times F_n$, where each $F_i$ comes equipped with embeddings $L \to F_i$ and 
$M \to F_i$ which yield a commutative diagram with the given embeddings $K \to L$ and $K \to M$. So we may consider the set $\mathcal{P}_i$ of all places on each $F_i$ which lie over both $P$ and $R$.
\end{notation}

The 
following identity is a direct consequence of Theorem~\ref{global} and Proposition~\ref{e&f}.

\begin{cor}\label{sum-e}
In the setting of Notation~\ref{note}, we have
\[\sum_{i=1}^n \sum_{Q\in \mathcal{P}_i} e(Q/S) = e(P/S) \cdot e(R/S).\]
\end{cor}

\begin{proof}
By Theorem~\ref{global} we have the degree identity 
\[
\sum_{i=1}^n \sum_{Q\in \mathcal{P}_i} [\widehat{F_{i,Q}} : \widehat{M}] = [\widehat{L} : \widehat{K}],
\]
which by Proposition~\ref{e&f} implies the following identity
\[
\sum_{i=1}^n \sum_{Q\in \mathcal{P}_i} e(Q/R) \cdot f(Q/R) = e(R/S) \cdot f(R/S). 
\]
Since the constant field $k$ is algebraically closed, all values $f(Q/R)$ and $f(R/S)$ equal $1$. 
Thus, the identity becomes
\[
\sum_{i=1}^n \sum_{Q\in \mathcal{P}_i} e(Q/R) = e(P/S),
\] 
which implies the desired identity 
\[ 
\sum_{i=1}^n \sum_{Q\in \mathcal{P}_i} e(Q/S) = e(P/S) \cdot e(R/S)
\] 
upon multiplying both sides by $e(R/S)$.
\end{proof}

\begin{cor}\label{num-p}
If Notation~\ref{note} holds and
in addition both $L/K$ and $M/K$ are separable, and also at least one of $Q$ or $R$ is tamely ramified over $S$, 
then \[\Bigl\lvert\coprod_{i=1}^n \mathcal{P}_i\Bigr\rvert = \gcd\bigl(e(P/S), e(R/S)\bigr).\]
\end{cor}

\begin{proof}
Each extension $F_i/K$ is finite separable, $L$ and $M$ are subextensions of $F_i$ over $K$ with $F_i = L.M$. 
By Corollary~\ref{old-abh} each place $Q$ in any $\mathcal{P}_i$ satisfies $e(Q/S)=\lcm\bigl(e(P/S),e(R/S)\bigr)$. Therefore, the identity $\abs{\coprod_{i=1}^n \mathcal{P}_i} = \gcd\bigl(e(P/S), e(R/S)\bigr)$ follows 
from Corollary~\ref{sum-e}.
\end{proof}

\begin{rmk}
In case both $Q$ and $R$ are tamely ramified over $S$, Corollary~\ref{num-p} becomes a corrected 
version of \cite[Lemma~7.1]{DF} which is another version of Abhyankar's lemma. 
\end{rmk}



\end{document}